\documentclass[11pt]{article}

\usepackage[centertags]{amsmath}
\usepackage{amsfonts,color}
\usepackage{amssymb,multirow}
\usepackage{comment}
\usepackage{amsthm,authblk}
\usepackage{graphicx,placeins,subcaption}
\usepackage{multirow,hyperref}

\newtheorem{theorem}{Theorem}[section]

\newtheorem{lemma}{Lemma}[section]

\numberwithin{equation}{section}

\begin{document}
\thispagestyle{empty}

\title{Toward Realistic Energy Forecasting: A Delay-Enhanced Fractional-Order Supply-Demand Model}
\author[1]{S. Naveen}
\author[2,3,*]{S. Noeiaghdam}
\affil[1]{Department of Mathematics, School of Arts, Sciences, Humanities \& Education, SASTRA Deemed University, Thanjavur-613401, Tamil Nadu, India. \url{snaveen9790@gmail.com}}
\affil[2]{Institute of Mathematics, Henan Academy of Sciences, Zhengzhou, 450046, China. \url{snoei@hnas.ac.cn}}
\affil[3]{Department of Mathematical Sciences, Saveetha School of Engineering, SIMATS, Chennai, 602105, India.}
\date{}
\maketitle

\begin{abstract}
In order to represent the complex dynamics of contemporary energy systems, a novel fractional-order energy supply-demand model with time delay is presented in this investigation. The fractional model, compared to traditional integer-order models, naturally accommodates for memory and genetic effects, and the incorporation of delay component takes into account unavoidable lags in energy production, transmission, and consumption. To ensure the mathematical rigor of the proposed model, we prove the existence and uniqueness of solutions. We additionally examine into the model's stability within the Ulam-Hyers concept and demonstrate that it is resilient to minor uncertainties and perturbations. To handle fractional derivatives with delay systems, we utilize the Grnwald-Letnikov (GL) discretization scheme, which offers a straightforward and effective method for approximating the solutions. The impact of delay parameters and fractional orders on system behavior is investigated numerically, demonstrating how they shape oscillations, convergence rates, and equilibrium states. According to the results, fractional-order modeling with delay, reinforced by the GL discretization scheme, provides a flexible and realistic framework for examining energy dynamics. This framework gives important insights for long-term policy planning, supply management, and demand forecasting. Additionally, the framework lays the groundwork for upcoming additions that incorporate optimization techniques, stochastic effects, and the integration of renewable energy sources, all of which will further the development of effective and sustainable energy systems. We also discuss three sensitivity analysis scenarios including high demand combined with low supply, high renewable and reduced and imports low demand with high imports which the results show stable and efficient results.

\textbf{\textit{keywords:}} Energy supply-demand dynamics, Fractional-order systems, Time-delay modeling, Ulam-Hyers stability, Grnwald-Letnikov.
\end{abstract}
	\section{Introduction}
	
	Mathematical models in the form of dynamical systems serve as a fundamental tool for describing and predicting the evolution of complex phenomena across time, with applications spanning engineering, bio-mathematics, and beyond. In engineering, dynamical systems are essential for modeling control processes, signal processing, and energy systems, where they facilitate optimization and stability analysis. In bio-mathematics, such models have been extensively used to study population dynamics, epidemiological patterns, and physiological processes, providing insights into nonlinear behaviors, oscillations, and bifurcations that govern biological systems. Moreover, their versatility extends to economics, physics, and environmental sciences, where they enable simulation, forecasting, and decision-making under uncertainty, underscoring their interdisciplinary significance \cite{1,2}. 

	Mathematical models of energy supply and demand systems, often formulated as nonlinear dynamical systems, play a critical role in analyzing the stability, control, and forecasting of resource utilization in complex environments. Such models capture the dynamic interactions between energy supply and demand, often exhibiting chaotic or oscillatory behaviors that require advanced control strategies. For instance, time-delayed feedback and adaptive control methods have been proposed to stabilize chaotic energy systems and achieve synchronization under uncertain parameters \cite{3,5,7}. 
Further studies have extended these models to higher-dimensional systems, such as four-dimensional energy supply-demand frameworks, where linear feedback and model reference control approaches have been developed to enhance stability and performance \cite{6,7}. Applications of these models continue to evolve with modern computational techniques, including physics-informed neural networks and fractional calculus, which provide more accurate simulations and predictions of nonlinear energy dynamics \cite{8,9}. In \cite{23} the authors are presented the variable order RLC circuit system for Mittag-Leffler kernel with integral boundary conditions. In \cite{24} the authors are discussed the Caputo fractional variable order derivative for RC and LC circuit system. The integration of such mathematical approaches contributes to improving the reliability, efficiency, and sustainability of energy resource management in increasingly complex systems \cite{4}. In \cite{21}, the authors are discussed a study on controllability of fractional dynamical systems with distributed delays modeled by $\Omega$-Hilfer fractional derivatives.

	Fractional-order models, grounded in fractional calculus, extend classical integer order differential and integral operators to non-integer (fractional) orders, enabling the modeling of systems with memory, hereditary properties, and anomalous dynamics. These models are especially powerful because the fractional derivative inherently incorporates nonlocality, reflecting system states over their entire history rather than solely at the present moment \cite{10}. The rule of fractional order refers to this generalization of differentiation and integration to arbitrary (real or complex) orders, often formalized through operators such as Riemann-Liouville and Caputo derivatives, each with unique mathematical and physical interpretations \cite{11}. In \cite{22,25} the authors are investigated the time delay for variable order Lorenz system and Chen system. The author discussed the variable order fractional enzyme kinetics model with time delay in \cite{26}. 
	
	Applications of fractional-order modeling span a wide range of scientific and engineering fields, offering advantages over classical integer-order approaches by incorporating memory and hereditary effects into system dynamics. In epidemiology, fractional-order models have been effectively applied to study the spread and control of infectious diseases such as COVID-19, where they capture more realistic dynamics of disease transmission and provide robust numerical solutions. For example, Noeiaghdam et al. developed a nonlinear fractional COVID-19 model with accuracy controlled by the CESTAC method and CADNA library \cite{13}, while Hedayati et al. proposed wavelet-based procedures for solving fractional epidemic models, enhancing computational efficiency \cite{15}. Similarly, Sivashankar et al. analyzed the stability of COVID-19 outbreaks using the Caputo-Fabrizio fractional derivative, demonstrating the capability of fractional calculus to describe epidemic patterns more accurately than classical models \cite{18}. 
Beyond epidemiology, fractional-order methods have also been employed in fuzzy modeling: Allahviranloo et al. introduced a fuzzy fractional approach based on generalized Taylor expansions to address fuzzy fractional differential equations \cite{14}. In applied mathematics and control theory, Ghomanjani and collaborators developed novel numerical techniques, such as Said Ball curves for solving fractional differential-algebraic equations \cite{16} and transcendental Bernstein polynomials for two-dimensional fractional optimal control problems \cite{17}, broadening the applicability of fractional calculus in optimization and control. Collectively, these works highlight the growing importance of fractional-order models not only in biological and epidemiological systems but also in numerical analysis and engineering applications, reinforcing their significance in advancing mathematical modeling \cite{12,18}. 

	In this paper, we discuss the energy supply and demand model
	\begin{equation}\label{1}
		\left\{
		\begin{array}{l}
			\displaystyle  X'_1=a_1X_1(1-X_1/\mathcal{W})-a_2X_2(X_2+X_3)-d_3X_4, \\
			\\
			\displaystyle  X'_2=-z_1X_2-z_2X_3+z_3X_1[N-(X_1-X_3)], \\
			\\
			\displaystyle  X'_3=s_1X_3(s_2X_1-s_3), \\
			\\
			\displaystyle  X'_4=d_1X_1-d_2X_4,
		\end{array}
		\right.
	\end{equation}
	where parameters $a_i, z_i, s_i, d_i, \mathcal{W}, N > 0$, are positive constants, and
	$N < \mathcal{W}$.  The energy resource demand of region $A$ is denoted by $X_1$, while the energy resource supply provided by region $B$ to $A$ is expressed as $X_2$. The amount of energy imported by region $A$ is represented by $X_3$, and the renewable energy resources available in region $A$ are denoted by $X_4$. The elasticity factor of region $A$'s energy demand is given by $a_1$. The supply factor of region $B$, which influences the energy demand of region $A$, is represented by $a_2$. The energy import coefficient of region $A$, which also affects its energy demand, is likewise denoted by $a_2$. The maximum possible energy demand in region $A$ is expressed by $\mathcal{W}$, while $N$ stands for the valve value. The coefficients describing the effects of region $B$'s supply to $A$, the energy imports of $A$, and the impact of region $A$'s demand on the supply rate of region $B$ are given by $z_1$, $z_2$, and $z_3$, respectively. The velocity factor of energy imports in region $A$ is denoted by $s_1$, the per-unit benefit of imported energy is expressed as $s_2$, and the cost of imported energy is indicated by $s_3$. The factor $d_1$ represents the influence of region $A$'s energy demand on the rate of renewable energy adoption, while $d_2$ measures the effect of renewable energy resources themselves on this adoption rate. Finally, $d_3$ reflects the influence of renewable energy resources on the energy demand of region $A$. When these parameters are assigned specific values, the system evolves into a chaotic state \cite{2,3}.
	$$
	a_1= 1;~ a_2= 0.6;~ z_1= 1;~ z_2 = 1.2;~ z_3 = 1.5;~ s_1 =
	11;~ s_2 = 0.7;~ s_3 = 0.6;$$
	$$ d_1= 0.1;~ d_2 = 0.7;~ d_3 = 0.8;~ \mathcal{W} = 3.8;~ N= 1.2.
	$$
	The fractional order form of the problem (\ref{1}) can be presented as
	\begin{equation}\label{2}
		\left\{
		\begin{array}{l}
			\displaystyle {}^CD^\alpha x_1(t)=a_1x_1(t)(1-x_1(t)/\mathcal{W})-a_2x_1(t)(x_2(t-\tau)+x_3(t))-d_3x_4(t-\tau), \\
			\\
			\displaystyle  {}^CD^\alpha x_2(t)=-z_1x_2(t)-z_2x_3(t)+z_3x_1(t)[N-(x_1(t)-x_3(t))], \\
			\\
			\displaystyle  {}^CD^\alpha x_3(t)=s_1x_3(t)(s_2x_1(t)-s_3), \\
			\\
			\displaystyle  {}^CD^\alpha x_4(t)=d_1x_1(t)-d_2x_4(t),
		\end{array}
		\right.
	\end{equation}
	where  ${}^CD^\alpha$  shows the of $\alpha$ order Caputo derivative \cite{19,20}. {Fig. \ref{f1} is the model graph to show the logic relations between functions and parameters. }

	\begin{figure}
		\centering
		\includegraphics[width=0.9\linewidth]{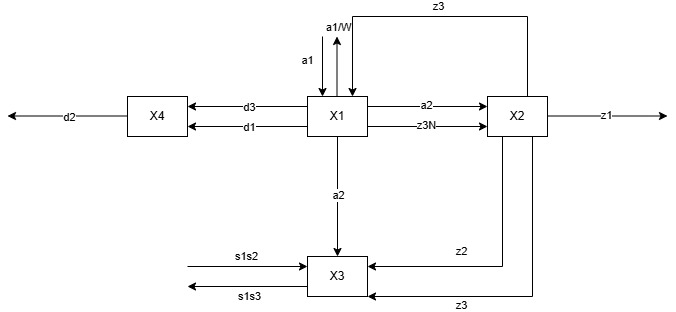}
		\caption{The graphical form of the model. }
		\label{f1}
	\end{figure}

	In this paper, we will focus on the fractional order model of the energy supply and demand with time delay. Fractional-order models of energy supply and demand with time delay provide a more realistic representation of energy systems by incorporating both memory effects (through fractional derivatives) and response lags (through time delay). The fractional order captures long-term dependencies in consumption and generation, while the time delay reflects inevitable practical factors such as transmission lags, decision-making delays, and storage-response times. Including time delay is crucial, as ignoring it can lead to inaccurate predictions, instability, or poor control performance, whereas accounting for it enables more reliable forecasting, stability analysis, and the design of effective control strategies for complex energy systems. Section 2, is to discuss the existence and uniqueness theories of the problem.   In Section 3, we prove the main theorem to show the stability and we can see that we have Ulam-Hyers stable. We have presented the numerical results in Section 4. We have various discussions for different values of the fractional order as well as with time delay.

	
	\section{Existence and Uniqueness Theory}
	
	In this section, we develop the existence and uniqueness theory for the proposed fractional-order energy supply-demand model with time delay. Establishing these results is essential, as they ensure that the model is mathematically consistent and produces well-defined solutions under the given assumptions. The existence of a solution guarantees that the system can realistically describe the underlying energy dynamics, while uniqueness rules out ambiguity, confirming that the model's behavior is determined solely by the initial conditions and parameters. These properties provide a rigorous theoretical foundation for the subsequent numerical analysis and practical applications of the model in energy systems. Moreover, we proved the existence and uniqueness theorem by using the Banach contraction fixed point theorem and Schauder fixed point theorem.\\
	Applying the Caputo derivative operator \cite{19,20} for the system (\ref{2}) we get
	\begin{align}\label{3.1}
		x_1(t)=&x_1(0)+\frac{1}{\Gamma(\alpha)}\int\limits_0^t(t-s)^{\alpha-1}{X}_1(s,x_1(s),x_2(s),x_3(s),x_4(s))ds,\nonumber\\
		x_2(t)=&x_2(0)+\frac{1}{\Gamma(\alpha)}\int\limits_0^t(t-s)^{\alpha-1}{X}_2(s,x_1(s),x_2(s),x_3(s),x_4(s))ds,\nonumber\\
		x_3(t)=&x_3(0)+\frac{1}{\Gamma(\alpha)}\int\limits_0^t(t-s)^{\alpha-1}{X}_3(s,x_1(s),x_2(s),x_3(s),x_4(s))ds,\\
		x_4(t)=&x_4(0)+\frac{1}{\Gamma(\alpha)}\int\limits_0^t(t-s)^{\alpha-1}{X}_4(s,x_1(s),x_2(s),x_3(s),x_4(s))ds.\nonumber
	\end{align}
	Where
	\begin{align}\label{3.2}
		X_1(t,x_1,x_2,x_3,x_4)=&a_1x_1(t)(1-x_1(t)/\mathcal{W})-a_2x_2(t)(x_2(t-\tau)+x_3(t))-d_3x_4(t-\tau),\nonumber\\
		X_2(t,x_1,x_2,x_3,x_4)=&-z_1x_2(t)-z_2x_3(t)+z_3x_1(t)[N-(x_1(t)-x_3(t))], \nonumber\\
		X_3(t,x_1,x_2,x_3,x_4)=&s_1x_3(t)(s_2x_1(t)-s_3), \\
		X_4(t,x_1,x_2,x_3,x_4)=&d_1x_1(t)-d_2x_4(t).\nonumber
	\end{align}
	
	Moreover, let us consider the closed and bounded interval $[0,T]$. We define the norm as follows:
	\begin{equation*}
		\lVert x_1,x_2,x_3,x_4 \rVert=\sup\limits_{t\in[0,T]}\lvert x_1(t)\rvert + \sup\limits_{t\in[0,T]}\lvert x_2(t)\rvert + \sup\limits_{t\in[0,T]}\lvert x_3(t)\rvert + \sup\limits_{t\in[0,T]}\lvert x_4(t)\rvert.
	\end{equation*}
	The system \eqref{3.1} can now be reformulated in the following integral form:
	\begin{equation}\label{3.3}
		X(t)=X(0)+\frac{1}{\Gamma(\alpha)}\int\limits_0^t(t-s)^{\alpha-1}G(s,X(s))\,ds,
	\end{equation}
	where
	\begin{align*}
		X(t)=(x_1,x_2,x_3,x_4)^\mathbb{T}, \qquad
		X(0)=(x_1(0),x_2(0),x_3(0),x_4(0))^\mathbb{T},
	\end{align*}
	and
	\begin{align*}
		G(s,X(s))=\begin{cases}
			X_1(t,x_1(t),x_2(t),x_3(t),x_4(t))\\
			X_2(t,x_1(t),x_2(t),x_3(t),x_4(t))\\
			X_3(t,x_1(t),x_2(t),x_3(t),x_4(t))\\
			X_4(t,x_1(t),x_2(t),x_3(t),x_4(t))
		\end{cases}.
	\end{align*}
	
	To establish the existence and uniqueness of solutions, we impose the following growth conditions on the function $G:[0,T]\times\mathbb{R}\to\mathbb{R}$:
	
	$(\mathbb{H}1)$ There exists a constant $\beta_1>0$ such that for any $X(t),\bar{X}(t)\in\mathbb{R}$,
	\begin{equation*}
		\lvert G(s,X(s))-G(s,\bar{X}(s)) \rvert \le \beta_1\lvert X(s)-\bar{X}(s)\rvert.
	\end{equation*}
	
	$(\mathbb{H}2)$ There exist constants $\beta_2,\beta_3>0$ such that
	\begin{equation*}
		\lvert G(s,X(s))\rvert \le \beta_2\lvert X(s)\rvert+\beta_3.
	\end{equation*}
	
	Based on these assumptions and using the Schauder fixed point theorem, we establish the following result.
	
	\begin{theorem}
		If $G$ is continuous and satisfies $(\mathbb{H}2)$, then the fractional model \eqref{2} admits at least one solution.
	\end{theorem}
	
	\begin{proof}
		Define the operator $\mathbb{A}:\mathbb{X}\to\mathbb{X}$ as
		\begin{equation}\label{6}
			\mathbb{A}(X)(t)=X(0)+\frac{1}{\Gamma(\alpha)}\int\limits_0^t(t-s)^{\alpha-1}G(s,X(s))\,ds.
		\end{equation}
		For each $X\in\mathbb{X}$, we estimate
		\begin{align*}
			\lvert \mathbb{A}(X)(t)\rvert &\le \lvert X_0\rvert+\int\limits_0^t (t-s)^{\alpha-1}\lvert G(s,X(s))\rvert \,ds\\
			&\le \lvert X_0\rvert+\int\limits_0^t (t-s)^{\alpha-1}\left[\beta_2\lvert X\rvert+\beta_3\right]ds\\
			&\le \lvert X_0\rvert+\frac{T^\alpha}{\Gamma(\alpha+1)}\left[\beta_2\lVert X\rVert+\beta_3\right].
		\end{align*}
		This further implies
		\begin{equation}\label{7}
			\lVert \mathbb{A}(X)\rVert\le\lvert X_0\rvert+\frac{T^\alpha}{\Gamma(\alpha+1)}\left[\beta_2\lVert X\rVert+\beta_3\right]\le\Re.
		\end{equation}
		From \eqref{7}, it follows that $X\in\mathbb{A}$, hence $X(\mathbb{A})\subset\mathbb{A}$. Therefore, $\mathbb{A}$ is a bounded operator.
		
		To prove complete continuity, consider $t_1<t_2\in[0,T]$. Then
		\begin{align}\label{8}
			\lvert \mathbb{A}(X)(t_2)- \mathbb{A}(X)(t_1)\rvert &= \Bigg\lvert\frac{1}{\Gamma(\alpha)}\int\limits_0^{t_2}(t_2-s)^{\alpha-1} G(s,X(s))\,ds
			-\frac{1}{\Gamma(\alpha)}\int\limits_0^{t_1}(t_1-s)^{\alpha-1} G(s,X(s))\,ds\Bigg\rvert\nonumber\\
			&\le \frac{1}{\Gamma(\alpha)}\Bigg[\int\limits_0^{t_1}(t_1-s)^{\alpha-1}(t_2-s)^{\alpha-1} \lvert G(s,X(s))\rvert ds
			+\int\limits_{t_1}^{t_2}(t_2-s)^{\alpha-1} \lvert G(s,X(s))\rvert ds \Bigg]\nonumber\\
			&\le \frac{(\beta_2+\Re+\beta_3)}{\Gamma(\alpha+1)}\left(t_2^\alpha-t_1^\alpha+2(t_2-t_1)\right).
		\end{align}
		From \eqref{8}, we observe that as $t_2\to t_1$, the right-hand side tends to zero. Hence
		\begin{equation*}
			\lVert \mathbb{A}(X)(t_2)- \mathbb{A}(X)(t_1)\rVert\to 0, \quad \text{as } t_2\to t_1.
		\end{equation*}
		Thus, $\mathbb{A}$ is equicontinuous. By the ArzelÃ -Ascoli theorem, $\mathbb{A}$ is completely continuous. Therefore, by Schauder's fixed point theorem, system \eqref{2} possesses at least one solution.
	\end{proof}
	
	We now prove the uniqueness of the solution for system \eqref{2}.
	
	\begin{theorem}
		The system \eqref{2} admits a unique solution if the following condition holds:
		\begin{equation*}
			\frac{T^\alpha}{\Gamma(\alpha+1)}\beta_4<1.
		\end{equation*}
	\end{theorem}
	
	\begin{proof}
		Consider the operator $\mathbb{A}:\mathbb{X}\to\mathbb{X}$ defined in \eqref{6}. For $X,\bar{X}\in\mathbb{X}$, we have
		\begin{align*}
			\lVert \mathbb{A}(X)-\mathbb{A}(\bar{X})\rVert &= \sup\limits_{t\in[0,T]}\Bigg\lvert \frac{1}{\Gamma(\alpha)}\int\limits_0^t(t-s)^{\alpha-1}\Big(G(s,X(s))-G(s,\bar{X}(s))\Big)\,ds\Bigg\rvert\\
			&\le \frac{T^\alpha}{\Gamma(\alpha+1)}\beta_4\lVert X-\bar{X}\rVert.
		\end{align*}
		Thus,
		\begin{equation}\label{9}
			\lVert \mathbb{A}(X)-\mathbb{A}(\bar{X})\rVert\le\frac{T^\alpha}{\Gamma(\alpha+1)}\beta_4\lVert X-\bar{X}\rVert.
		\end{equation}
		From \eqref{9}, it follows that $\mathbb{A}$ is a contraction operator. Hence, by Banach's fixed point theorem, system \eqref{2} has a unique solution.
	\end{proof}

	
	\section{Stability Analysis}
	
	In this section, we focus on the stability analysis of the proposed fractional-order energy supply-demand model with time delay. Stability plays a central role in validating the reliability of a mathematical model, as it ensures that small deviations in initial conditions, external disturbances, or model perturbations do not lead to unrealistic or unbounded behaviors in the system's solutions. By establishing Ulam-Hyers stability, we guarantee that the model is not only mathematically sound but also robust under practical uncertainties, making it suitable for applications in real-world energy systems where fluctuations and external influences are inevitable. This analysis therefore provides a solid foundation for the trustworthiness of the numerical results and further applications of the model.

	Now, in order to study the stability analysis of the model \eqref{2}, we proceed by introducing a small perturbation in $\Phi\in[0,T]$, where $\Phi(0)=0$, which depends solely on the solution $X$ as follows:
	\begin{itemize}
		\item $\lvert \Phi(t)\rvert\le\epsilon$, for any $\epsilon>0$,
		\item ${}^CD^\alpha X(t)=G(t,X(t))+\Phi(t)$.
	\end{itemize}
	
	\begin{lemma}\label{lem1}
		For the perturbed problem
		\begin{equation}
			{}^CD^\alpha X(t)=G(t,X(t))+\Phi(t),
		\end{equation}
		with the initial condition $X(0)=X_0$, the corresponding solution satisfies the following inequality:
		\begin{equation}\label{10}
			\left| X(t)-\left(X_0+\frac{1}{\Gamma(\alpha)}\int\limits_0^t(t-s)^{\alpha-1}G(s,X(s))ds\right)\right|
			\le\frac{T^\alpha}{\Gamma(\alpha+1)}\epsilon=\Pi_{T,\alpha\epsilon}.
		\end{equation}
	\end{lemma}
	
	\begin{proof}
		The proof of Lemma \ref{lem1} follows directly in a similar manner as the preceding results, hence it is omitted here for brevity.
	\end{proof}
	
	\begin{theorem}
		Under the assumption $(H1)$ together with condition \eqref{10}, the solution of the integral equation \eqref{3.3} is Ulam-Hyers stable. Accordingly, the numerical solutions of the model \eqref{2} preserve Ulam-Hyers stability provided that the following condition is fulfilled:
		\begin{equation}\label{11}
			\Lambda=\frac{T^\alpha}{\Gamma(\alpha+1)}\beta_4<1.
		\end{equation}
	\end{theorem}
	
	\begin{proof}
		Let $X\in\mathbb{X}$ be any solution, and let $\bar{X}$ denote at most one solution of \eqref{3.3}. Then we have
		\begin{align}\label{12}
			\lvert X(t)-\bar{X}(t)\rvert
			=&\left| X(t)-\left(X_0+\frac{1}{\Gamma(\alpha)}\int\limits_0^t(t-s)^{\alpha-1}G(s,\bar{X}(s))ds\right)\right|\nonumber\\
			\le& \left| X(t)-\left(X_0+\frac{1}{\Gamma(\alpha)}\int\limits_0^t(t-s)^{\alpha-1}G(s,\bar{X}(s))ds\right)\right|\\
			&+\left| \frac{1}{\Gamma(\alpha)}\int\limits_0^t(t-s)^{\alpha-1}G(s,\bar{X}(s))ds-\frac{1}{\Gamma(\alpha)}\int\limits_0^t(t-s)^{\alpha-1}G(s,\bar{X}(s))ds\right|\nonumber.
		\end{align}
		Finally, by applying condition \eqref{11} and taking the sup norm on both sides of \eqref{12}, we establish that the solution of the integral equation \eqref{3.3} is Ulam-Hyers stable. Therefore, it follows that the solution of the original model \eqref{2} is also Ulam-Hyers stable.
	\end{proof}

	\section{Numerical Approximation}
	
	To obtain approximate solutions of the proposed fractional-order supply-demand system with time delay, we employ the Grnwald-Letnikov (GL) discretization scheme. This approach is widely used in fractional calculus due to its simplicity and its ability to approximate fractional derivatives directly from their definition as infinite series. By discretizing the time domain with a uniform step size, the GL method generates a sequence of coefficients that capture the memory effect inherent in fractional-order systems. The method is well suited for delay differential equations since delays can be incorporated by shifting the state variables by the corresponding step indices. In what follows, we present a step-by-step implementation of the Grnwald-Letnikov scheme tailored to our fractional-order delay system, including the initialization, computation of coefficients, assignment of initial conditions, and the main iteration procedure for updating the system states. Fig. \ref{Alg} shows the main algorithm:

	\begin{figure}[h]
		\centering
		\includegraphics[width=0.9\linewidth]{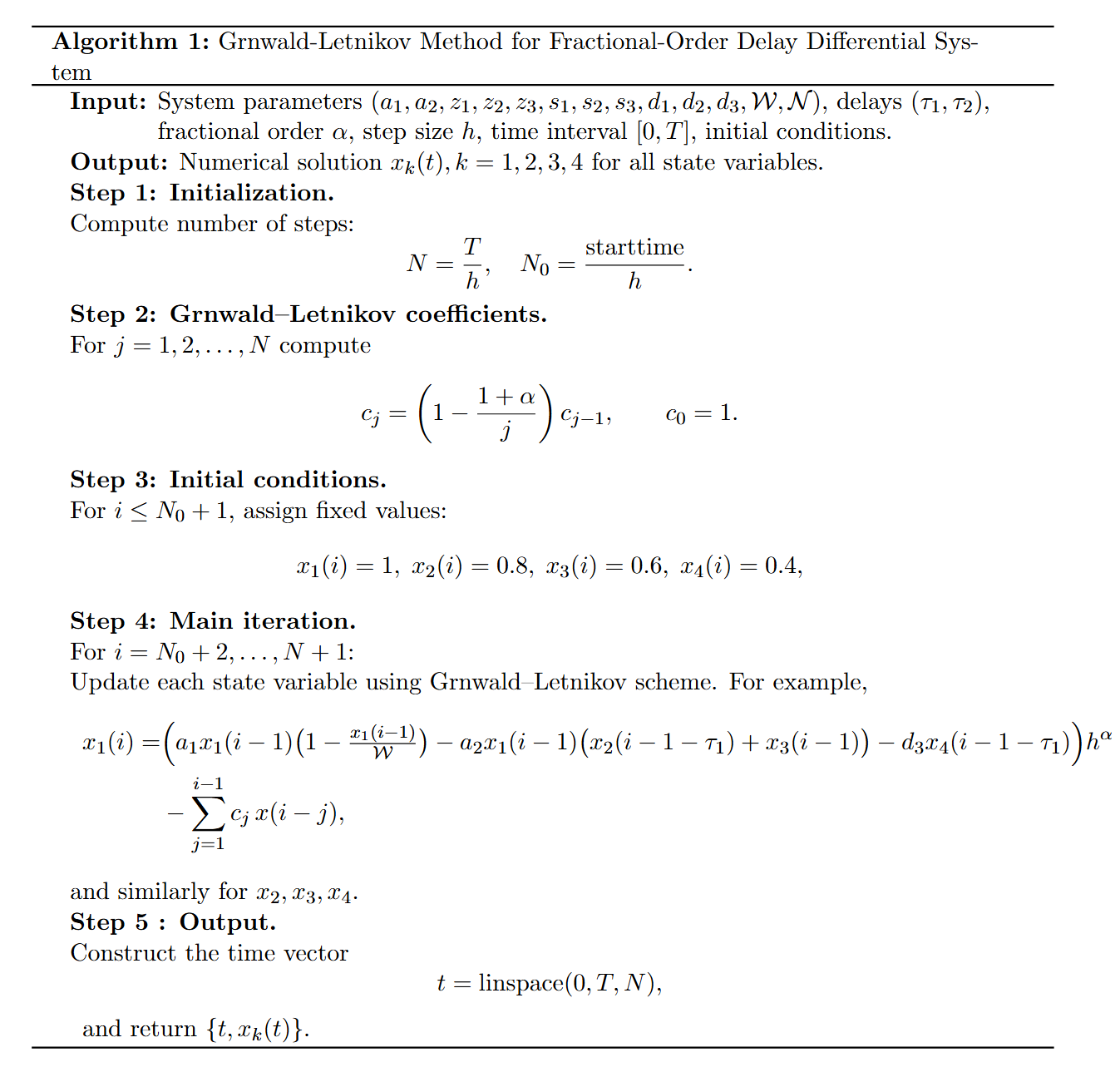}
		\caption{Numerical Approximation Algorithm.  }
		\label{Alg}
	\end{figure}

\FloatBarrier

\section{Numerical Results and Discussion}

In this section, we present the numerical results and discussion of the proposed fractional-order energy supply-demand model with time delay. The main objective is to investigate how variations in the fractional differentiation order and delay parameter influence the system's dynamic behavior, including oscillations, convergence, and long-term stability of the state variables. Such analysis is important because fractional-order models capture memory and hereditary effects, while time delays represent realistic lags in energy production, distribution, and consumption processes. By examining different combinations of fractional orders and delays, we can better understand the interplay between these factors, assess the robustness of the model, and provide insights into how energy systems respond under practical operating conditions.

	\begin{figure}[h]
		\centering
		\includegraphics[width=0.9\linewidth]{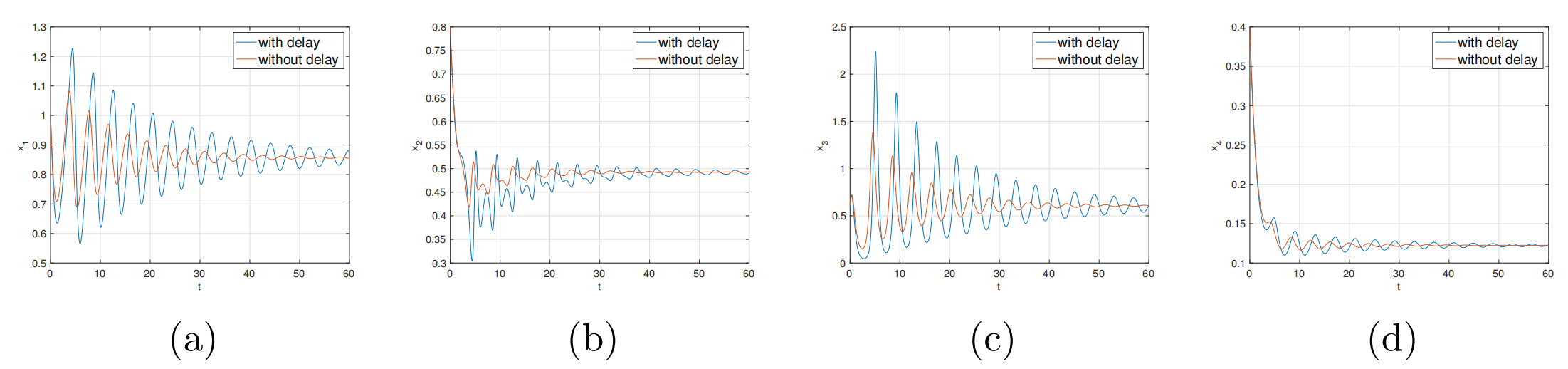}
		\caption{Fractional order proposed model with order 0.98 and delay 0.5 and without delay. }
		\label{fig1}
	\end{figure}

\FloatBarrier
\begin{table}[h]
	\centering
	\caption{Fractional order proposed system with order $0.98.$}
	\label{tab1}
	\begin{tabular}{|c|cccc|cccc|}
		\hline
		\multirow{2}{*}{{Time}} &
		\multicolumn{4}{c|}{{With   Delay}} &
		\multicolumn{4}{c|}{{Without   Delay}} \\ \cline{2-9}
		&
		\multicolumn{1}{c|}{\textbf{$x_1$}} &
		\multicolumn{1}{c|}{\textbf{$x_2$}} &
		\multicolumn{1}{c|}{\textbf{$x_3$}} &
		\textbf{$x_4$} &
		\multicolumn{1}{c|}{\textbf{$x_1$}} &
		\multicolumn{1}{c|}{\textbf{$x_2$}} &
		\multicolumn{1}{c|}{\textbf{$x_3$}} &
		\textbf{$x_4$} \\ \hline
		0 &
		\multicolumn{1}{c|}{1} &
		\multicolumn{1}{c|}{0.8} &
		\multicolumn{1}{c|}{0.6} &
		0.4 &
		\multicolumn{1}{c|}{1} &
		\multicolumn{1}{c|}{0.8} &
		\multicolumn{1}{c|}{0.6} &
		0.4 \\ \hline
		10 &
		\multicolumn{1}{c|}{0.6249} &
		\multicolumn{1}{c|}{0.4342} &
		\multicolumn{1}{c|}{0.7901} &
		0.1238 &
		\multicolumn{1}{c|}{0.7797} &
		\multicolumn{1}{c|}{0.4735} &
		\multicolumn{1}{c|}{0.3848} &
		0.1168 \\ \hline
		20 &
		\multicolumn{1}{c|}{0.9587} &
		\multicolumn{1}{c|}{0.458} &
		\multicolumn{1}{c|}{0.3384} &
		0.1198 &
		\multicolumn{1}{c|}{0.8797} &
		\multicolumn{1}{c|}{0.4955} &
		\multicolumn{1}{c|}{0.7588} &
		0.1255 \\ \hline
		30 &
		\multicolumn{1}{c|}{0.7734} &
		\multicolumn{1}{c|}{0.4945} &
		\multicolumn{1}{c|}{0.7314} &
		0.1229 &
		\multicolumn{1}{c|}{0.8587} &
		\multicolumn{1}{c|}{0.4894} &
		\multicolumn{1}{c|}{0.5433} &
		0.1212 \\ \hline
		40 &
		\multicolumn{1}{c|}{0.9115} &
		\multicolumn{1}{c|}{0.4815} &
		\multicolumn{1}{c|}{0.5384} &
		0.1226 &
		\multicolumn{1}{c|}{0.8519} &
		\multicolumn{1}{c|}{0.494} &
		\multicolumn{1}{c|}{0.6317} &
		0.1227 \\ \hline
		50 &
		\multicolumn{1}{c|}{0.8202} &
		\multicolumn{1}{c|}{0.4938} &
		\multicolumn{1}{c|}{0.6122} &
		0.1217 &
		\multicolumn{1}{c|}{0.8608} &
		\multicolumn{1}{c|}{0.4921} &
		\multicolumn{1}{c|}{0.6006} &
		0.1224 \\ \hline
		60 &
		\multicolumn{1}{c|}{0.8807} &
		\multicolumn{1}{c|}{0.4906} &
		\multicolumn{1}{c|}{0.6162} &
		0.1231 &
		\multicolumn{1}{c|}{0.8554} &
		\multicolumn{1}{c|}{0.4928} &
		\multicolumn{1}{c|}{0.6073} &
		0.1224 \\ \hline
	\end{tabular}
\end{table}\FloatBarrier

Figure \ref{fig1} compares the dynamics of the system with and without a delay of $\tau=0.5$. For all four state variables, the trajectories exhibit oscillatory behavior in the transient stage before stabilizing at steady-state values. The presence of delay increases both the amplitude and duration of oscillations. In particular, the energy demand ($x_1$) and energy supply ($x_2$) converge more slowly in the delayed case. The energy import ($x_3$) shows sharp peaks during the initial period, while renewable energy ($x_4$) rapidly settles to a small constant level. These results suggest that the introduction of a delay temporarily destabilizes the energy system but does not prevent its long-term stability. In Table \ref{tab1} represented the non-integer order proposed model with order $0.98$ for with and without delay.\\

Figure \ref{fig2} shows the system response with a smaller delay of $\tau=0.1$. Compared with Figure \ref{fig1}, the oscillations in all state variables are significantly reduced, and convergence to equilibrium is faster. Both energy demand ($x_1$) and supply ($x_2$) oscillate moderately before stabilizing, while the energy import ($x_3$) shows a smoother decay after an initial overshoot. The renewable energy component ($x_4$) again stabilizes quickly at a low constant value. This indicates that a small delay only induces mild disturbances, and the overall system retains better synchronization between demand, supply, and imports. In Table \ref{tab2} represented the non-integer order proposed model with order $0.99$ for with and without delay.
\begin{table}[h]
	\centering
	\caption{Fractional order proposed system with order $0.99$}
	\label{tab2}
	\begin{tabular}{|c|cccc|cccc|}
		\hline
		\multirow{2}{*}{{Time}} &
		\multicolumn{4}{c|}{{With   Delay}} &
		\multicolumn{4}{c|}{{Without   Delay}} \\ \cline{2-9}
		&
		\multicolumn{1}{c|}{\textbf{$x_1$}} &
		\multicolumn{1}{c|}{\textbf{$x_2$}} &
		\multicolumn{1}{c|}{\textbf{$x_3$}} &
		\textbf{$x_4$} &
		\multicolumn{1}{c|}{\textbf{$x_1$}} &
		\multicolumn{1}{c|}{\textbf{$x_2$}} &
		\multicolumn{1}{c|}{\textbf{$x_3$}} &
		\textbf{$x_4$} \\ \hline
		0 &
		\multicolumn{1}{c|}{1} &
		\multicolumn{1}{c|}{0.8} &
		\multicolumn{1}{c|}{0.6} &
		0.4 &
		\multicolumn{1}{c|}{1} &
		\multicolumn{1}{c|}{0.8} &
		\multicolumn{1}{c|}{0.6} &
		0.4 \\ \hline
		10 &
		\multicolumn{1}{c|}{0.5914} &
		\multicolumn{1}{c|}{0.4211} &
		\multicolumn{1}{c|}{1.035} &
		0.1279 &
		\multicolumn{1}{c|}{0.755} &
		\multicolumn{1}{c|}{0.4662} &
		\multicolumn{1}{c|}{0.3556} &
		0.1156 \\ \hline
		20 &
		\multicolumn{1}{c|}{0.8992} &
		\multicolumn{1}{c|}{0.4579} &
		\multicolumn{1}{c|}{0.1671} &
		0.1128 &
		\multicolumn{1}{c|}{0.9138} &
		\multicolumn{1}{c|}{0.4914} &
		\multicolumn{1}{c|}{0.8201} &
		0.1274 \\ \hline
		30 &
		\multicolumn{1}{c|}{0.8354} &
		\multicolumn{1}{c|}{0.5197} &
		\multicolumn{1}{c|}{1.3256} &
		0.1325 &
		\multicolumn{1}{c|}{0.8464} &
		\multicolumn{1}{c|}{0.487} &
		\multicolumn{1}{c|}{0.484} &
		0.1198 \\ \hline
		40 &
		\multicolumn{1}{c|}{0.8524} &
		\multicolumn{1}{c|}{0.4738} &
		\multicolumn{1}{c|}{0.2831} &
		0.1152 &
		\multicolumn{1}{c|}{0.8522} &
		\multicolumn{1}{c|}{0.4964} &
		\multicolumn{1}{c|}{0.6875} &
		0.1237 \\ \hline
		50 &
		\multicolumn{1}{c|}{0.8526} &
		\multicolumn{1}{c|}{0.509} &
		\multicolumn{1}{c|}{1.0181} &
		0.1287 &
		\multicolumn{1}{c|}{0.8653} &
		\multicolumn{1}{c|}{0.4904} &
		\multicolumn{1}{c|}{0.5713} &
		0.122 \\ \hline
		60 &
		\multicolumn{1}{c|}{0.8637} &
		\multicolumn{1}{c|}{0.4801} &
		\multicolumn{1}{c|}{0.3782} &
		0.1179 &
		\multicolumn{1}{c|}{0.8501} &
		\multicolumn{1}{c|}{0.4938} &
		\multicolumn{1}{c|}{0.6202} &
		0.1225 \\ \hline
	\end{tabular}
\end{table}\FloatBarrier

\FloatBarrier

	\begin{figure}[h]
		\centering
		\includegraphics[width=0.9\linewidth]{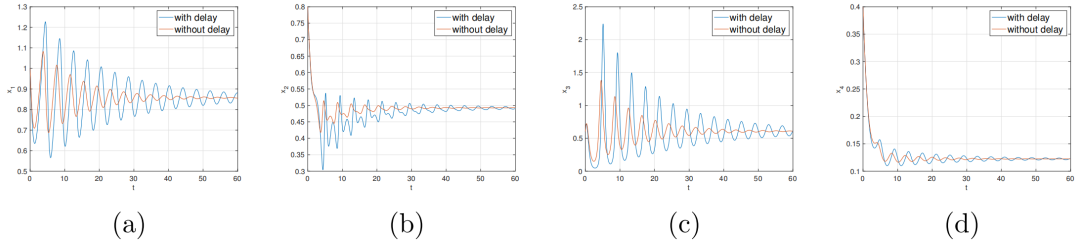}
		\caption{Fractional order proposed model with order 0.98 and delay 0.1 }
		\label{fig2}
	\end{figure}

The system response with an intermediate delay of $0.3$ is presented in Figure \ref{fig3}. In this figure, oscillations are more pronounced than in Figure \ref{fig2} but less severe compared to Figure \ref{fig1}. The demand ($x_1$) and supply ($x_2$) variables exhibit fluctuations before stabilization, and the energy import ($x_3$) experiences higher peaks than under delay $\tau=0.1$. Nevertheless, the renewable energy ($x_4$) stabilizes quickly, showing limited sensitivity to the delay value. This highlights that moderate delay introduces noticeable instability in the short term, but the system still remains bounded and converges to equilibrium. In Table \ref{tab3} represented the integer order proposed model for with and without delay.
\begin{table}[h]
	\centering
	\caption{Integer order proposed system}
	\label{tab3}
	\begin{tabular}{|c|cccc|cccc|}
		\hline
		\multirow{2}{*}{Time} &
		\multicolumn{4}{c|}{With   Delay} &
		\multicolumn{4}{c|}{Without   Delay} \\ \cline{2-9}
		&
		\multicolumn{1}{c|}{$x_1$} &
		\multicolumn{1}{c|}{$x_2$} &
		\multicolumn{1}{c|}{$x_3$} &
		$x_4$ &
		\multicolumn{1}{c|}{$x_1$} &
		\multicolumn{1}{c|}{$x_2$} &
		\multicolumn{1}{c|}{$x_3$} &
		$x_4$ \\ \hline
		0 &
		\multicolumn{1}{c|}{1} &
		\multicolumn{1}{c|}{0.8} &
		\multicolumn{1}{c|}{0.6} &
		0.4 &
		\multicolumn{1}{c|}{1} &
		\multicolumn{1}{c|}{0.8} &
		\multicolumn{1}{c|}{0.6} &
		0.4 \\ \hline
		10 &
		\multicolumn{1}{c|}{0.5749} &
		\multicolumn{1}{c|}{0.427} &
		\multicolumn{1}{c|}{1.5976} &
		0.1358 &
		\multicolumn{1}{c|}{0.723} &
		\multicolumn{1}{c|}{0.4544} &
		\multicolumn{1}{c|}{0.3264} &
		0.1144 \\ \hline
		20 &
		\multicolumn{1}{c|}{0.689} &
		\multicolumn{1}{c|}{0.4148} &
		\multicolumn{1}{c|}{0.0824} &
		0.1036 &
		\multicolumn{1}{c|}{0.988} &
		\multicolumn{1}{c|}{0.4684} &
		\multicolumn{1}{c|}{0.8099} &
		0.1296 \\ \hline
		30 &
		\multicolumn{1}{c|}{1.1354} &
		\multicolumn{1}{c|}{0.3655} &
		\multicolumn{1}{c|}{0.1503} &
		0.123 &
		\multicolumn{1}{c|}{0.8008} &
		\multicolumn{1}{c|}{0.4808} &
		\multicolumn{1}{c|}{0.4078} &
		0.1172 \\ \hline
		40 &
		\multicolumn{1}{c|}{0.7302} &
		\multicolumn{1}{c|}{0.5349} &
		\multicolumn{1}{c|}{2.2332} &
		0.1418 &
		\multicolumn{1}{c|}{0.8884} &
		\multicolumn{1}{c|}{0.497} &
		\multicolumn{1}{c|}{0.82} &
		0.1268 \\ \hline
		50 &
		\multicolumn{1}{c|}{0.6653} &
		\multicolumn{1}{c|}{0.4122} &
		\multicolumn{1}{c|}{0.1182} &
		0.1056 &
		\multicolumn{1}{c|}{0.8527} &
		\multicolumn{1}{c|}{0.4863} &
		\multicolumn{1}{c|}{0.4704} &
		0.1197 \\ \hline
		60 &
		\multicolumn{1}{c|}{1.1177} &
		\multicolumn{1}{c|}{0.3778} &
		\multicolumn{1}{c|}{0.1697} &
		0.1227 &
		\multicolumn{1}{c|}{0.8467} &
		\multicolumn{1}{c|}{0.4982} &
		\multicolumn{1}{c|}{0.7207} &
		0.1242 \\ \hline
	\end{tabular}
\end{table}\FloatBarrier

	\begin{figure}[h]
		\centering
		\includegraphics[width=0.9\linewidth]{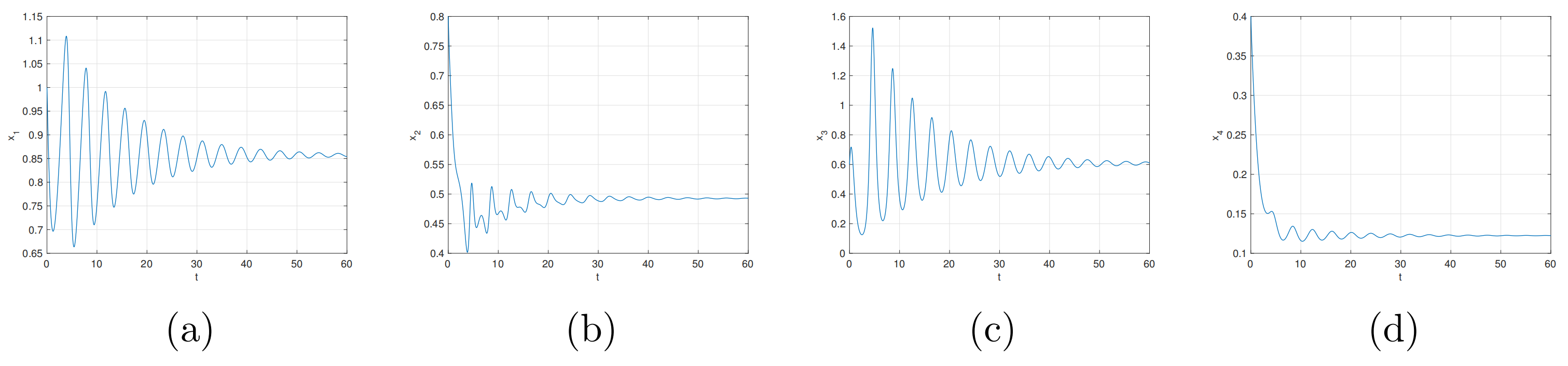}
		\caption{Fractional order proposed model with order 0.98 and delay 0.3 }
		\label{fig3}
	\end{figure}

\FloatBarrier

	\begin{figure}[h]
		\centering
		\includegraphics[width=0.9\linewidth]{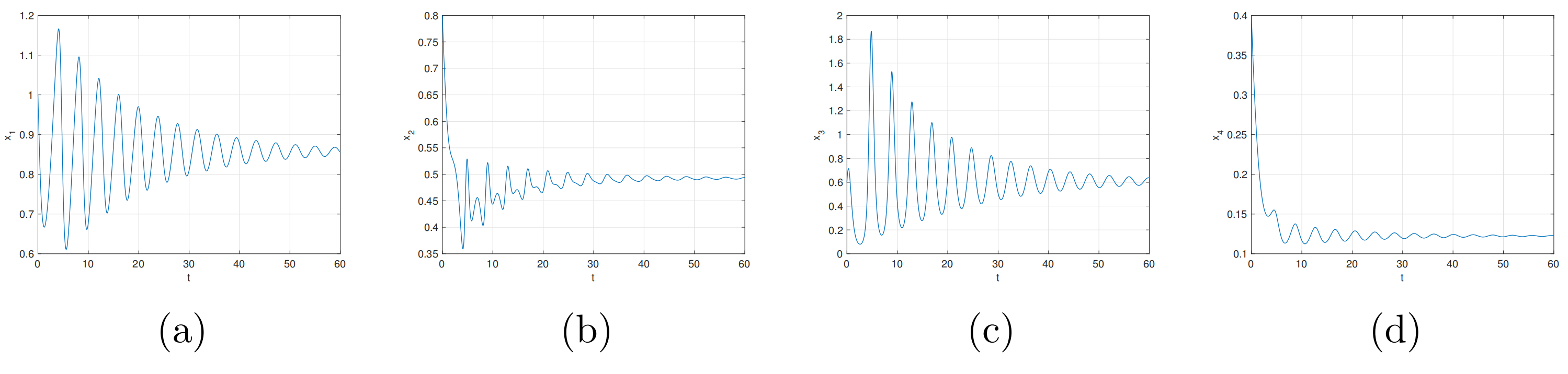}
		\caption{Fractional order proposed model with order 0.98 and delay 0.5 }
		\label{fig4}
	\end{figure}

\FloatBarrier
Figure \ref{fig4} illustrates the system behavior when the fractional order is $0.98$ and the delay is fixed at $\tau=0.5$. In this case, all four state variables exhibit pronounced oscillations before gradually stabilizing. The energy demand ($x_1$) in Figure \ref{fig4}- a initially oscillates with large amplitude but slowly decays and approaches a steady level near $0.85$. The energy supply ($x_2$) in Figure \ref{fig4}- b shows strong fluctuations at the beginning, dropping below $0.35$ and then recovering, before converging towards approximately $0.5$. The energy import ($x_3$) in Figure \ref{fig4}- c  experiences sharp peaks in the early stage, reaching values above $2.0$, but eventually settles down to a lower constant level. The renewable energy component ($x_4$) in Figure \ref{fig4}- d displays rapid damping and stabilizes quickly near $0.12$, with relatively small oscillations compared to the other variables.

These results confirm that under a higher delay of $0.5$, the system exhibits more instability and stronger oscillatory dynamics in demand, supply, and imports. Nevertheless, all state variables remain bounded and eventually converge to equilibrium, with renewable energy showing the fastest stabilization.

Figures \ref{fig5}-\ref{fig7} illustrate the dynamical behavior of the system for different fractional orders ($q = 0.98, 0.99, 1$) and time delays  ($\tau = 0.5, 0.3, 0.1$). Each figure consists of four subplots corresponding to the time evolution of the state variables $x_1, x_2, x_3,$ and $x_4$. The comparison among different fractional orders demonstrates how even small variations in the derivative order influence system stability, oscillations, and convergence rates.

In Figure \ref{fig5}, with the largest delay $\tau = 0.5$, the system exhibits
pronounced oscillations across all state variables. Figures \ref{fig5}- a, b, and c show that the oscillation amplitudes are higher for fractional order $q = 0.98$ compared
to $0.99$ and $1$, indicating that decreasing the fractional order increases the
system's oscillatory behavior. Figure \ref{fig5}- d highlights that although the system
eventually stabilizes, the larger delay results in slow damping and extended transient
behavior.

	\begin{figure}[h]
		\centering
		\includegraphics[width=0.9\linewidth]{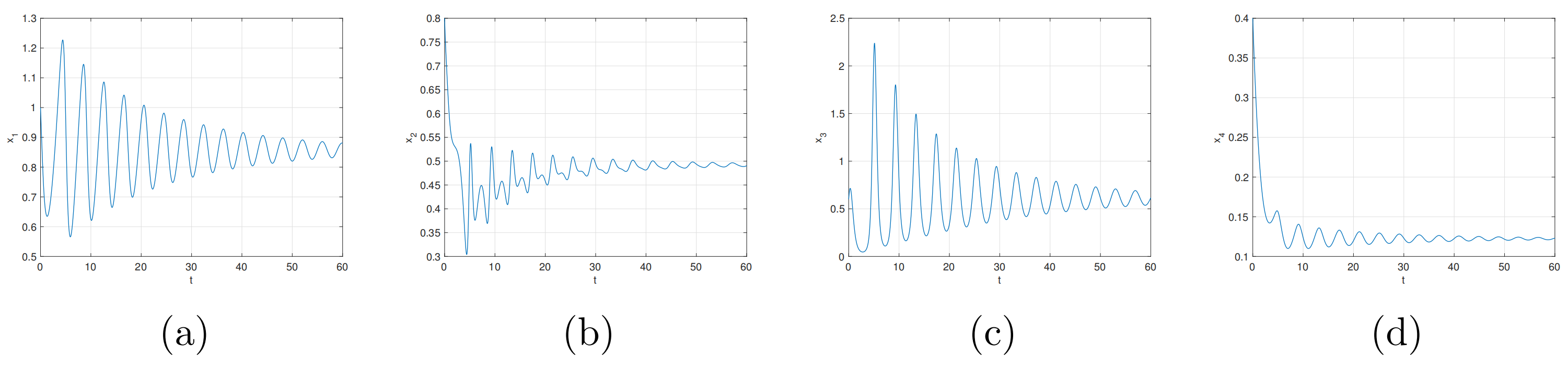}
		\caption{Various fractional orders $0.98$, $0.99$ and $1$ with time delay $\tau=0.5$. }
		\label{fig5}
	\end{figure}

\FloatBarrier
\subsection{{Comparison Between Integer-Order Model and Fractional-Order Model With Time Delay}}
The integer-order energy supply-demand model without delay offers a simplified representation of energy interactions by assuming that system responses occur instantaneously without any memory of past states. This approach enables straightforward mathematical analysis and provides basic insights into the relationship between energy demand, supply, imports, and renewable resources. However, it overlooks the inherent inertia and time-dependent adjustments present in real-world energy systems, where factors such as policy implementation, production scaling, and consumption behavior evolve gradually over time. As a result, the integer-order model often fails to capture the delayed and cumulative effects that influence long-term system dynamics.

The fractional-order energy supply-demand model with time delay addresses these limitations by introducing both memory and delay characteristics into the system dynamics. The fractional derivative represents the hereditary nature of energy processes, allowing the present state to depend on the entire history of the system. The time delay term, meanwhile, captures the practical time lag between decision-making and its impact on supply or demand. Together, these features yield a more realistic and flexible framework capable of modeling slow transients, oscillations, and long-term dependencies that are typically observed in actual energy networks.

When the time delay is reduced to $\tau = 0.3$, as shown in Figure \ref{fig6}, the oscillatory behavior weakens compared to $\tau = 0.5$. Sub-figure \ref{fig6}- a, b, c illustrate that fractional order $0.98$ still produces oscillations, but with lower amplitudes and faster decay than in Figure \ref{fig5}. For the integer order case ($1$), the system stabilizes more rapidly, confirming that both higher derivative orders and smaller time delays improve convergence. Figure \ref{fig6}- d further confirms that reducing delay enhances damping and shortens the transient period.

In Figure \ref{fig7}, with the smallest time delay $\tau = 0.1$, the system dynamics become  much smoother. Sub-figure \ref{fig7}- a, b and c reveal that oscillations are significantly suppressed, and all fractional orders lead to rapid stabilization. The differences between $0.98, 0.99,$ and $1$ become minimal, indicating that at very small delays, the effect
of fractional differentiation order on stability is less pronounced. Figure \ref{fig7}- d 
shows almost immediate convergence to the steady state, confirming strong robustness against
delays at this scale.

	\begin{figure}[h]
		\centering
		\includegraphics[width=0.9\linewidth]{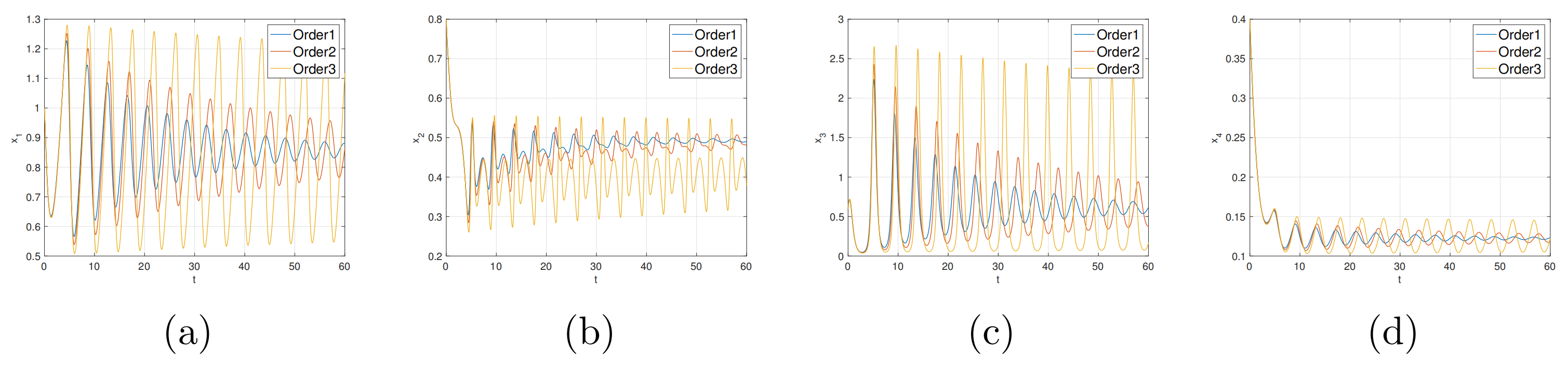}
		\caption{Various fractional orders $0.98$, $0.99$ and $1$ with time delay $\tau=0.3$. }
		\label{fig6}
	\end{figure}

\FloatBarrier

In Figure \ref{fig7}, with the smallest time delay of 0.1, the system exhibits much smoother behavior. Figures \ref{fig7}- a,  b and c reveal that oscillations are significantly suppressed, and the system stabilizes rapidly for all fractional orders. The difference between orders 0.98, 0.99, and 1 becomes less prominent, showing that at small delays, the effect of fractional differentiation order on stability is weaker. Figure \ref{fig7}- d demonstrates almost immediate decay to a steady state, indicating strong stability and robustness of the system for small delays.

The results in Figures \ref{fig5}-\ref{fig7} highlight the combined influence of fractional order and time delay on system dynamics. Lower fractional orders ($0.98$) induce stronger oscillations and slower convergence, while higher orders ($0.99,\ 1$) enhance stability. Similarly, larger delays ($\tau = 0.5$) prolong oscillations and transient phases, whereas smaller delays ($\tau = 0.1$) lead to faster damping and quicker stabilization. These findings emphasize that both fractional differentiation order and delay must be carefully considered when analyzing or designing such dynamical systems.

	\begin{figure}[h]
		\centering
		\includegraphics[width=0.9\linewidth]{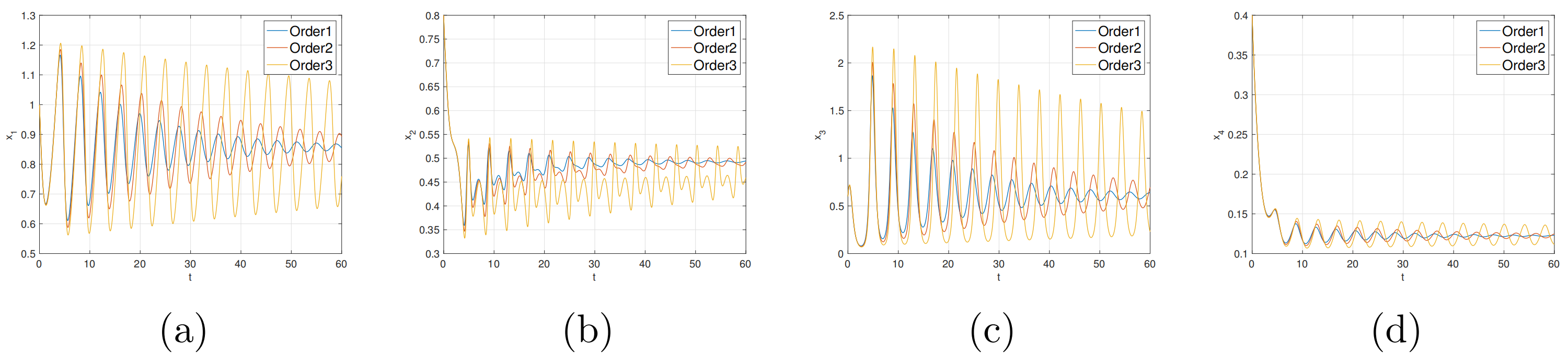}
		\caption{Various fractional orders $0.98$, $0.99$ and $1$ with time delay $\tau=0.1$. }
		\label{fig7}
	\end{figure}

\FloatBarrier
{Overall, while the integer-order model without delay is simpler and computationally efficient, the fractional-order model with time delay provides a more comprehensive and accurate description of real energy dynamics. Its ability to capture memory effects and delayed responses enhances both predictive power and stability analysis, offering valuable insights for energy policy formulation and sustainable system design. Hence, the fractional-order delayed framework represents a significant advancement in modeling the complex temporal behavior of energy supply-demand systems.}\\

Figure \ref{fig8} presents the system responses for the fractional order $0.98$ under different time delays $\tau = 0.5, 0.3, 0.1$. The sub figures correspond to the state variables $x_1, x_2, x_3,$ and $x_4$. It can be observed that all state trajectories initially exhibit oscillatory behavior before converging to steady-state values. Larger delays result in slightly higher oscillation amplitudes in the transient stage, but eventually, all responses converge to the same equilibrium point. This demonstrates that the system retains stability even under varying delays
for fractional order $0.98$.

	\begin{figure}[h]
		\centering
		\includegraphics[width=0.9\linewidth]{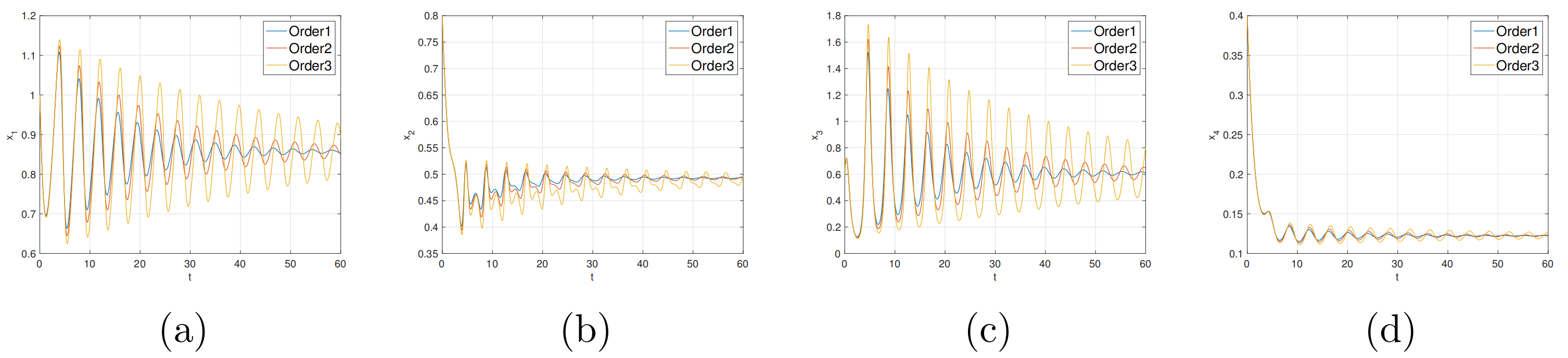}
		\caption{Fractional order $0.98$ with different time delays $\tau=0.5,\ 0.3,\ 0.1$. }
		\label{fig8}
	\end{figure}

\FloatBarrier

Figure \ref{fig9} shows the system dynamics for the fractional order $0.99$ with the same time delays. Compared to Figure \ref{fig8}, the oscillations are more sustained, and the damping is slower. This indicates that as the fractional order approaches unity, the memory effect of the system increases, leading to longer-lasting oscillatory transients. Nevertheless, the system remains stable, and the trajectories corresponding to different delays overlap in the long term, highlighting robustness against delay variations.

	\begin{figure}[h]
		\centering
		\includegraphics[width=0.9\linewidth]{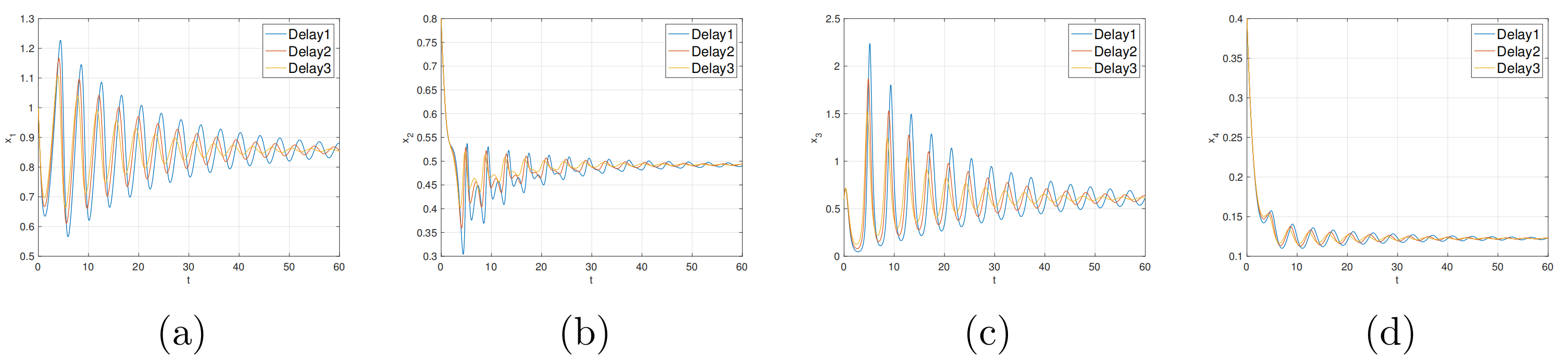}
		\caption{Fractional order $0.99$ with different time delays $\tau=0.5,\ 0.3,\ 0.1$.}
		\label{fig9}
	\end{figure}

\FloatBarrier

Figure \ref{fig10} illustrates the responses of the integer-order system with delays $\tau = 0.5, 0.3, 0.1$. In this case, the transient oscillations are more prominent compared to the fractional-order systems, and the convergence to equilibrium is relatively slower. The influence of time delay is also clearer, particularly in the initial oscillations, where the trajectory with the largest delay shows the most pronounced deviation. However, similar to the fractional-order cases, all trajectories ultimately converge to the same steady-state, confirming overall system stability.

A comparison of Figures \ref{fig8}-\ref{fig10} reveals that fractional-order models, especially with $0.98$, provide faster damping and smoother convergence compared to both the higher fractional order $0.99$ and the integer-order case. This highlights the effectiveness of fractional-order modeling in improving system stability and reducing the impact of time delays.

	\begin{figure}[h]
		\centering
		\includegraphics[width=0.9\linewidth]{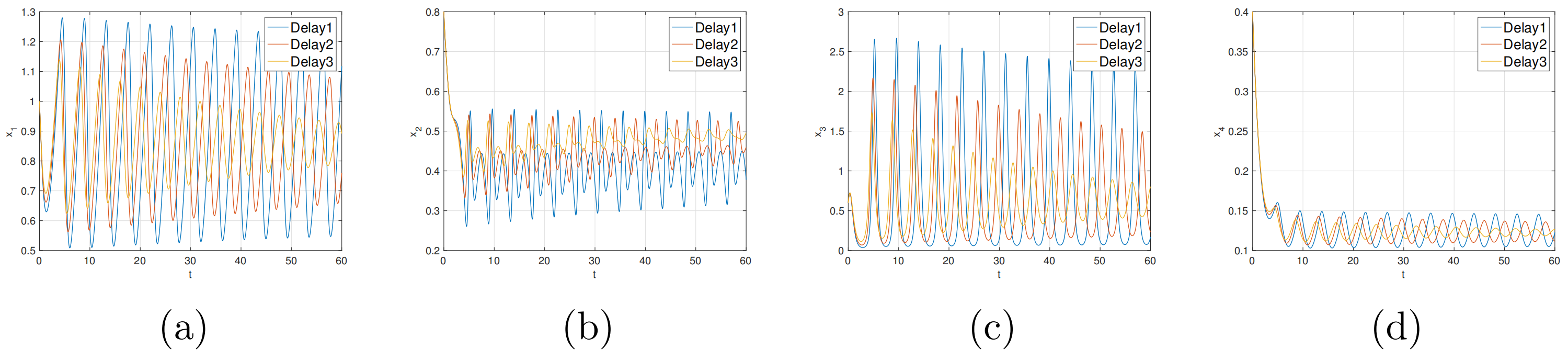}
		\caption{Integer order with different time delays $\tau=0.5,\ 0.3,\ 0.1$.}
		\label{fig10}
	\end{figure}

\FloatBarrier


\subsection{Sensitivity Analysis Scenarios}
We discuss the following scenarios. 


\paragraph{Scenario A: High Demand Combined With Low Supply (Energy Crisis).}
A demand surge coincides with external supply reduction.
\begin{itemize}
	\item Increase $x_1(0)$ by 25\%.
	\item Decrease $x_2(0)$ by 30\%.
	\item Increase $z_2$ to emphasize the importance of imports.
	\item Increase $d_3$ to strengthen the mitigating effect of renewables on demand.
\end{itemize}
This scenario represents a realistic crisis producing strong nonlinear oscillations.

\paragraph{Scenario B: High Renewables and Reduced Imports.}
Region~A moves toward energy independence.
\begin{itemize}
	\item Increase $x_4(0)$ by 50\%.
	\item Decrease $x_3(0)$ by 20\%.
	\item Lower $s_1$ and increase $d_1$.
\end{itemize}
This scenario evaluates system stability under renewable-dominated conditions.

\paragraph{Scenario C: Low Demand With High Imports (Oversupply).}
A recession reduces demand while long-term contracts force high import volumes.
\begin{itemize}
	\item Decrease $x_1(0)$ by 20\%.
	\item Increase $x_3(0)$ by 20\%.
	\item Decrease $a_1$ and increase $s_1$.
\end{itemize}
This scenario tests oversupply-induced oscillatory behavior.

	\begin{figure}[h]
		\centering
		\includegraphics[width=0.9\linewidth]{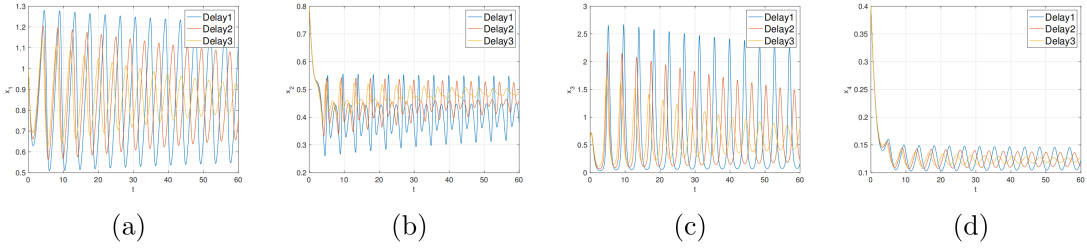}
		\caption{Sensitivity analysis of fractional order 0.98 with time delays $\tau=0.1$. }
		\label{fig11}
	\end{figure}

\FloatBarrier
Figure~\ref{fig11} illustrates the sensitivity analysis of the fractional-order energy supply--demand system for fractional order $\alpha = 0.98$ and time delay $\tau = 0.1$ under three representative scenarios. Figures \ref{fig11}- a, b, c and d show the temporal evolution of the state variables $x_1(t)$, $x_2(t)$, $x_3(t)$, and $x_4(t)$, respectively. The baseline dynamics are compared with Scenarios~A--C in order to examine the influence of coordinated perturbations in initial conditions and key system parameters.

In Scenario~A (high demand combined with low supply), the system exhibits pronounced transient oscillations, particularly in the demand and import components. The increased oscillation amplitudes observed in $x_1(t)$ and $x_3(t)$ reflect the destabilizing effects of demand surges and supply contraction during energy crisis conditions. Although the solutions remain bounded, the convergence toward equilibrium is slower than in the baseline case.

Scenario~B (high renewables and reduced imports) displays comparatively smoother trajectories across all state variables. The enhanced contribution of renewables introduces stronger damping effects, resulting in reduced oscillation amplitudes and faster stabilization. This scenario highlights the stabilizing role of renewable penetration in mitigating system fluctuations.

In contrast, Scenario~C (low demand with high imports) leads to sustained oscillatory behavior, most notably in the import dynamics $x_3(t)$. This response is indicative of oversupply-induced instabilities arising from delayed feedback mechanisms in the fractional-order system.

\section{Conclusion}
In this study, we have developed and analyzed a fractional-order energy supply-demand model incorporating time delays to better capture the dynamic characteristics of real-world energy systems. By extending the classical framework through fractional calculus, the model reflects memory and hereditary properties, which are essential for understanding the cumulative and long-term effects in energy production, distribution, and consumption. The inclusion of time delay further enhances the realism of the model, since energy processes are rarely instantaneous, and delays are naturally present in decision-making, storage, transmission, and demand response mechanisms. To support the theoretical formulation, we established existence and uniqueness results for the proposed system, ensuring that the model is mathematically well-posed and capable of producing reliable solutions. Moreover, the stability analysis, particularly under the Ulam-Hyers framework, confirmed that the model is robust against small perturbations, highlighting its reliability in practical applications where uncertainties and fluctuations are inevitable.

The results of this work underscore the importance of fractional-order modeling in energy system analysis. Unlike traditional integer-order approaches, the fractional framework allows us to incorporate complex dynamics, such as memory and hereditary effects, that are indispensable for describing real processes. Similarly, the time delay element introduces an additional layer of realism by accounting for inevitable lags in the supply chain and consumption cycle. Together, these features provide a more accurate representation of the energy sector, making the model suitable for exploring long-term policies, forecasting demand-supply balance, and optimizing system efficiency. The mathematical results obtained here not only justify the theoretical foundations but also provide a solid background for future numerical and simulation studies.

In practical terms, this research offers significant potential for application in modern energy management, particularly in renewable and smart grid systems, where delays and memory effects play a central role. The findings can aid policymakers, engineers, and system designers in predicting demand patterns, designing efficient load-leveling strategies, and mitigating instabilities caused by supply fluctuations. Furthermore, the framework established here can be extended to incorporate additional features such as stochastic influences, fuzzy uncertainties, or network effects, which are increasingly relevant in today's interconnected energy landscape.

To conclude, the proposed fractional-order time-delay energy supply-demand model provides a novel and effective mathematical tool for analyzing the intricate behavior of energy systems. It not only enriches the theoretical literature on fractional dynamical systems but also offers practical insights for addressing real-world energy challenges. Future research directions may include the integration of optimization techniques, numerical simulations with real data, and the development of computational algorithms tailored to fractional and delayed models. Such advancements will further bridge the gap between mathematical modeling and practical implementation, paving the way toward more sustainable, stable, and efficient energy systems.

\section*{Declarations}
\subsection*{Conflicts of interest} The authors declare no conflict of interest.

\subsection*{Ethical Statement}
This article does not contain any studies with human participants or animals performed by any of the authors.
\subsection*{Funding Statement}
The work of S. Noeiaghdam was funded by the High-Level Talent Research Start-up Project Funding of Henan Academy of Sciences (Project No. 241819246).
\subsection*{Data availability}
All data generated or analyzed during this study are included in this article.
\subsection*{Authors contributions}
S.N and S.N wrote the main manuscript text, edited, software and coding and prepared the results. All authors reviewed the manuscript.


\begin{thebibliography}{}

\end{thebibliography}


\begin{thebibliography}{70}
	
\bibitem{1} J.D. Murray, Mathematical Biology: I. An Introduction, Springer Science \& Business Media, 17 (2007).

\bibitem{2} M.W. Hirsch, S. Smale, R.L. Devaney, Differential Equations, Dynamical Systems, and an Introduction to Chaos, Academic Press (2013).

\bibitem{3} M. Sun, L. Tian, J. Xu, Time-delayed feedback control of the energy resource chaotic system, International Journal of Nonlinear Science, 1(3) (2006) 172-177.

\bibitem{4} M. Sun, L. Tian, Y. Fu, An energy resources demand--supply system and its dynamical analysis, Chaos, Solitons \& Fractals, 32(1) (2007) 168-180.

\bibitem{5} M. Sun, L. Tian, Q. Jia, Adaptive control and synchronization of a four-dimensional energy resources system with unknown parameters, Chaos, Solitons \& Fractals, 39(4) (2009) 1943-1949.

\bibitem{6} M. Sun, Y. Tao, X. Wang, L. Tian, The model reference control for the four-dimensional energy supply-demand system, Applied Mathematical Modelling, 35(10) (2011) 5165-5172.

\bibitem{7} M. Sun, Q. Jia, L. Tian, A new four-dimensional energy resources system and its linear feedback control, Chaos, Solitons \& Fractals, 39(1) (2009) 101-108.

\bibitem{8} V.T. Vo, S. Noeiaghdam, D. Sidorov, A. Dreglea, L. Wang, Solving Nonlinear Energy Supply and Demand System Using Physics-Informed Neural Networks, Computation, 13 (2025) 13. https://doi.org/10.3390/computation13010013

\bibitem{9} S. Noeiaghdam, D. Sidorov, Caputo-Fabrizio Fractional Derivative to Solve the Fractional Model of Energy Supply-Demand System, Mathematical Modelling of Engineering Problems, 7(3) (2020) 359-367. https://doi.org/10.18280/mmep.070305

\bibitem{10} M.A. Matlob, Y. Jamali, The concepts and applications of fractional order differential calculus in modeling of viscoelastic systems: A primer, Critical Reviews in Biomedical Engineering, 47(4) (2019).

\bibitem{11} J.L. Suzuki, M. Gulian, M. Zayernouri, M. D'Elia, Fractional modeling in action: A survey of nonlocal models for subsurface transport, turbulent flows, and anomalous materials, Journal of Peridynamics and Nonlocal Modeling, 5(3) (2023) 392-459.

\bibitem{12} A.M. Lopes, L. Chen, Fractional order systems and their applications, Fractal and Fractional, 6(7) (2022) 389.

\bibitem{13} S. Noeiaghdam, S. Micula, J.J. Nieto, Novel Technique to Control the Accuracy of a Nonlinear Fractional Order Model of COVID-19: Application of the CESTAC Method and the CADNA Library, Mathematics, 9 (2021) 1321. https://doi.org/10.3390/math9121321

\bibitem{14} T. Allahviranloo, Z. Noeiaghdam, S. Noeiaghdam, J.J. Nieto, A Fuzzy Method for Solving Fuzzy Fractional Differential Equations Based on the Generalized Fuzzy Taylor Expansion, Mathematics, 8(12) (2020) 2166. https://doi.org/10.3390/math8122166

\bibitem{15} M. Hedayati, R. Ezzati, S. Noeiaghdam, New Procedures of a Fractional Order Model of Novel Coronavirus (COVID-19) Outbreak via Wavelets Method, Axioms, 10 (2021) 122. https://doi.org/10.3390/axioms10020122

\bibitem{16} F. Ghomanjani, S. Noeiaghdam, Application of Said Ball curve for solving fractional differential-algebraic equations, Mathematics, 9 (2021) 1926. https://doi.org/10.3390/math9161926

\bibitem{17} F. Ghomanjani, S. Noeiaghdam, S. Micula, Application of transcendental Bernstein polynomials for solving two-dimensional fractional optimal control problems, Complexity, Article ID 4303775 (2022). https://doi.org/10.1155/2022/4303775

\bibitem{18} M. Sivashankar, S. Sabarinathan, V. Govindan, U. Fernandez-Gamiz, S. Noeiaghdam, Stability analysis of COVID-19 outbreak using Caputo-Fabrizio fractional differential equation, AIMS Mathematics, 8(2) (2023) 2720-2735. https://doi.org/10.3934/math.2023143

\bibitem{19} M. Caputo, M. Fabrizio, A new definition of fractional derivative without singular kernel, Progress in Fractional Differentiation \& Applications, 1(2) (2015) 73-85.

\bibitem{20} J. Losada, J.J. Nieto, Properties of a new fractional derivative without singular kernel, Progress in Fractional Differentiation \& Applications, 1(2) (2015) 87-92.

\bibitem{21} S. Jose, S. Naveen, V. Parthiban, A study on controllability of fractional dynamical systems with distributed delays modeled by $\Omega$-Hilfer fractional derivatives, International Journal of Dynamics and Control, 12(1) (2024) 259-270.

\bibitem{22} S. Naveen, V. Parthiban, Existence, uniqueness and error analysis of variable-order fractional Lorenz system with various type of delays, International Journal of Bifurcation and Chaos, 34(12) (2024) 2450152.

\bibitem{23} S. Naveen, K. Venkatachalam, V. Parthiban, Analysis of variable-order derivative with Mittag--Leffler kernel and integral boundary conditions for RLC circuit system, International Journal of Computer Mathematics (2025) 1-18.

\bibitem{24} S. Naveen, V. Parthiban, Variable-order Caputo derivative of LC and RC circuits system with numerical analysis, International Journal of Circuit Theory and Applications, 53(5) (2025) 3136-3156.

\bibitem{25} S. Naveen, V. Parthiban, Qualitative analysis of variable-order fractional differential equations with constant delay, Mathematical Methods in the Applied Sciences, 47(4) (2024) 2981-2992.

\bibitem{26} K. Agilan, S. Naveen, S. Suganya, V. Parthiban, Analysis of variable-order fractional enzyme kinetics model with time delay, Scientific Reports, 15(1) (2025) 34255.
	
	
\end{thebibliography}
\end{document}